\documentclass[a4paper,12pt]{amsart}

\usepackage{amsmath,amssymb,amsfonts,%amsthm,amsopn,
bbm}
\usepackage{graphicx,color}

\usepackage{amsmath}

\usepackage{float}
\usepackage{caption}
\usepackage{geometry}
\newtheorem{theorem}{Theorem}[section]
\newtheorem{lemma}[theorem]{Lemma}
\newtheorem{corollary}[theorem]{Corollary}
\newtheorem{proposition}[theorem]{Proposition}

\theoremstyle{remark}
\newtheorem{remark}[theorem]{Remark}

\newtheorem*{remark*}{Remark}

\theoremstyle{definition}

\newcommand\realp{\mathop{Re}}
\newcommand\dH{\,d{\mathcal H}^1}

\def\bR{\mathbb{R}}
\def\bC{\mathbb{C}}

\newcommand\cB{\mathcal{B}}
\newcommand\cA{\mathcal{A}}
\newcommand\cF{\mathcal{F}}
\newcommand\cS{\mathcal{S}}

\newcommand\cH{\mathcal{H}}

\newcommand\cV{\mathcal{V}}
\newcommand\bN{\mathbb{N}}

\newcommand{\PhiOmega}[1]{\Phi_\Omega(#1)}
\newcommand{\PhiOm}{\Phi_\Omega}

\newcommand\Aa{{\mathcal{A}_\alpha}}

\numberwithin{equation}{section}
\title{A Faber-Krahn inequality for Wavelet transforms}
\author{Jo\~ao P. G. Ramos and Paolo Tilli}
\begin{document}
\maketitle
\begin{abstract}
For some special window functions $\psi_{\beta} \in H^2(\bC^+),$ we prove that, over all sets $\Delta \subset \bC^+$ of fixed hyperbolic measure $\nu(\Delta),$ the ones over which the Wavelet transform  $W_{\overline{\psi_{\beta}}}$ with window $\overline{\psi_{\beta}}$ concentrates optimally are exactly the discs with respect to the pseudohyperbolic metric of the upper half space.  This answers a question raised by Abreu and D\"orfler in \cite{AbreuDoerfler}.

Our techniques make use of a framework recently developed by F. Nicola and the second author in \cite{NicolaTilli}, but in the hyperbolic context induced by the dilation symmetry of the Wavelet transform. This leads us naturally to use a hyperbolic rearrangement function, as well as the hyperbolic isoperimetric inequality, in our analysis.
\end{abstract}
\section{Introduction}

In this paper, our main focus will be to answer a question by L. D. Abreu and M. D\"orfler \cite{AbreuDoerfler} on the sets which maximise concentration of certain wavelet transforms. 

To that extent, given a fixed function $g \in L^2(\bR),$ the \emph{Wavelet transform} with window $g$ is defined as
\begin{equation}\label{eq:wavelet-transform}
W_gf(x,s) = \frac{1}{s^{1/2}} \int_{\bR} f(t)\overline{ g\left( \frac{t-x}{s}\right) }\, dt,
\quad \forall f \in L^2(\bR).
\end{equation}
This map is well-defined pointwise for each $x \in \bR, s > 0,$ but in fact, it has better properties if we restrict ourselves to certain subspaces of $L^2.$ Indeed, if $f,g$ are so that $\widehat{f},\widehat{g} = 0$ over the negative half line $(-\infty,0),$ then it can be shown that the wavelet transform is an isometric inclusion from $H^2(\bC^+)$ to $L^2(\bC^+,s^{-2} \, dx \, ds).$

This operator has been introduced first by I. Daubechies and T. Paul in \cite{DaubechiesPaul}, where the authors discuss its properties with respect to time-frequency localisation, in comparison to the short-time Fourier transform operator introduced previously by Daubechies in \cite{Daubechies} and Berezin \cite{Berezin}. Together with the short-time Fourier transform, the Wavelet transform has attracted attention of several authors. As the literature of this topic is extremely rich and we could not, by any means, provide a complete account of it here, we mention specially those interested in the problem of obtaining information from a domain from information on its localisation operator - see, for instance, \cite{AbreuDoerfler,AbreuSpeckbacher1, AbreuSpeckbacher2, AbreuGrochRomero, AbreuPerRomero, GroechenigBook, WongWaveletBook} and the references therein.

In this manuscript, we shall be interested in the continuous wavelet transform for certain special window functions, and how much of its mass, in an  $L^2(\bC^+,s^{-2} \, dx \, ds)-$sense, can be concentrated on certain subsets of the upper half space.

To that extent, fix $\beta > 0.$ We then define $\psi_{\beta} \in L^2(\bR)$ to be such that \[
\widehat{\psi_{\beta}}(t) = \frac{1}{c_{\beta}} 1_{[0,+\infty)} t^{\beta} e^{-t},
\]
where one lets $c_{\beta} = \int_0^{\infty} t^{2\beta - 1} e^{-2t} dt = 2^{2\beta -1}\Gamma(2\beta).$ Here, we normalise the Fourier transform as
\[
\widehat{f}(\xi) = \frac{1}{(2\pi)^{1/2}} \int_{\bR} f(t) e^{-it \xi} \, d \xi.
\]
Fix now a subset $\Delta \subset \bC^+$ of the upper half space.
We define then
\[
C_{\Delta}^{\beta} := \sup \left\{ \int_{\Delta} |W_{\overline{\psi_{\beta}}} f(x,s)|^2 \,\frac{ dx \, ds}{s^2} \colon f \in H^2(\bC^+), \|f\|_2 = 1 \right\}.
\]
The constant $C_{\Delta}^{\beta}$ measures, in some sense, the maximal wavelet concentration of order $\beta >0$ in $\Delta$. Fix then $\beta > 0.$ A natural question, in this regard, is that of providing sharp bounds for $C_{\Delta}^{\beta},$ in terms of some quantitative constraint additionally imposed on the set $\Delta.$ This problem has appeared previously in some places in the literature, especially in the context of the short-time Fourier transform \cite{AbreuSpeckbacher1, AbreuSpeckbacher2, NicolaTilli}. For the continuous wavelet transform, we mention, in particular, the paper by L. D. Abreu and M. D\"orfler \cite{AbreuDoerfler}, where the authors pose this question explicitly in their last remark.

The purpose of this manuscript is, as previously mentioned, to solve such a problem, under the contraint that the \emph{hyperbolic measure} of the set $\Delta$, given by
\[
\nu(\Delta) = \int_{\Delta} \frac{dx\, ds}{s^2} < +\infty,
\]
is \emph{prescribed}. This condition arises in particular if one tries to analyse when the localisation operators associated with $\Delta$
\[
P_{\Delta,\beta} f = ( (W_{\overline{\psi_{\beta}}})^{*} 1_{\Delta} W_{\overline{\psi_{\beta}}} ) f
\]
are bounded from $L^2$ to $L^2.$ One sees, by  \cite[Propositions~12.1~and~12.12]{WongWaveletBook}, that
\begin{equation}\label{eq:localisation-operator}
\| P_{\Delta,\beta} \|_{2 \to 2} \le \begin{cases}
													1, & \text{ or } \cr
													 \left(\frac{\nu(D)}{c_{\beta}}\right). & \cr
													\end{cases}
\end{equation}
As we see that
\[
C_{\Delta}^{\beta} = \sup_{f \colon \|f\|_2 = 1} \int_{\Delta} |W_{\overline{\psi_{\beta}}} f(x,s)|^2 \, \frac{ dx \, ds}{s^2} = \sup_{f \colon \|f\|_2 = 1} \langle P_{\Delta,\beta} f, f \rangle_{L^2(\bR)},
\]
we have the two possible bounds for $C_{\Delta}^{\beta},$ given by the two possible upper bounds in \eqref{eq:localisation-operator}. By considering the first bound, one is led to consider the problem of maximising $C_{\Delta}^{\beta}$ over all sets $\Delta \subset \bC^{+},$ which is trivial by taking $\Delta = \bC^+.$

From the second bound, however, we are induced to consider the problem we mentioned before. In this regard, the main result of this note may be stated as follows:

\begin{theorem}\label{thm:main} It holds that
	\begin{equation}\label{eq:first-theorem}
   C_{\Delta}^{\beta} \le C_{\Delta^*}^{\beta},
	\end{equation}
	where $\Delta^* \subset \bC^+$ denotes any pseudohyperbolic disc so that $\nu(\Delta) = \nu(\Delta^*).$ Moreover, equality holds in \eqref{eq:first-theorem} if and only if $\Delta$ is a pseudohyperbolic disc of measure $\nu(\Delta).$
\end{theorem}

The proof of Theorem \ref{thm:main} is inspired by the recent proof of the Faber-Krahn inequality for the short-time Fourier transform, by F. Nicola and the second author \cite{NicolaTilli}. Indeed, in the present case, one may take advantage of the fact that the wavelet transform induces naturally a mapping from $H^2(\bC^+)$ to analytic functions with some decay on the upper half plane. This parallel is indeed the starting point of the proof of the main result in \cite{NicolaTilli}, where the authors show that the short-time Fourier transform with Gaussian window induces naturally the so-called \emph{Bargmann transform}, and one may thus work with analytic functions in a more direct form.

The next steps follow the general guidelines as in \cite{NicolaTilli}: one fixes a function and considers certain integrals over level sets, carefully adjusted to match the measure constraints. Then one uses rearrangement techniques, together with a coarea formula argument with the isoperimetric inequality stemming from the classical theory of elliptic equations, in order to prove bounds on the growth of such quantities.

The main differences in this context are highlighted by the translation of our problem in terms of Bergman spaces of the disc, rather than Fock spaces. Furthermore, we use a rearrangement with respect to a \emph{hyperbolic} measure, in contrast to the usual Hardy--Littlewood rearrangement in the case of the short-time Fourier transform. This presence of hyperbolic structures induces us, further in the proof, to use the hyperbolic isoperimetric inequality. In this regard, we point out that a recent result by A. Kulikov \cite{Kulikov} used a similar idea in order to analyse extrema of certain monotone functionals on Hardy spaces. \\

This paper is structured as follows. In Section 2, we introduce notation and the main concepts needed for the proof, and perform the first reductions of our proof. With the right notation at hand, we restate Theorem \ref{thm:main} in more precise form - which allows us to state crucial additional information on the extremizers of inequality \eqref{eq:first-theorem} - in Section 3, where we prove it. Finally, in Section 4, we discuss related versions of the reduced problem, and remark further on the inspiration for the hyperbolic measure constraint in Theorem \ref{thm:main}. \\

\noindent\textbf{Acknowledgements.} J.P.G.R. would like to acknowledge financial support by the European Research Council under the Grant Agreement No. 721675 ``Regularity and Stability in Partial Differential Equations (RSPDE)''.

\section{Notation and preliminary reductions} Before moving on to the proof of Theorem \ref{thm:main}, we must introduce the notion which shall be used in its proof. We refer the reader to the excellent exposition in \cite[Chapter~18]{WongWaveletBook} for a more detailed account of the facts presented here.

\subsection{The wavelet transform} Let $f \in H^2(\bC^+)$ be a function on the Hardy space of the upper half plane. That is, $f$ is holomorphic on $\bC^+ = \{ z \in \bC \colon \text{Im}(z) > 0\},$ and such that
\[
\sup_{s > 0} \int_{\bR} |f(x+is)|^2 \, dx < +\infty.
\]
Functions in this space may be identified in a natural way with functions $f$ on the real line, so that $\widehat{f}$ has support on the positive line $[0,+\infty].$ We fix then a function $g \in H^2(\bC^+) \setminus \{0\}$ so that
\[
\| \widehat{g} \|_{L^2(\bR^+,t^{-1})}^2 < +\infty.
\]
Given a fixed $g$ as above, the \emph{continuous Wavelet transform} of $f$ with respect to the window $g$ is defined to be
\begin{equation}\label{eq:wavelet-def}
	W_gf(z) = \langle f, \pi_z g \rangle_{H^2(\bC^+)}
\end{equation}
where $z = x + i s,$ and $\pi_z g(t) = s^{-1/2} g(s^{-1}(t-x)).$ From the definition, it is not difficult to see that $W_g$ is an \emph{isometry} from  $H^2(\bC^+)$ to $L^2(\bC^+, s^{-2} \, dx \, ds),$ as long as $\| \widehat{g} \|_{L^2(\bR^+,t^{-1})}^2 = 1.$ \\

\subsection{Bergman spaces on $\bC^+$ and $D$}For every $\alpha>-1$,
the Bergmann space $\Aa(D)$ of the disc is the Hilbert space of all functions
$f:D\to \bC$ which are holomorphic in the unit disk $D$ and are such that
\[
\Vert f\Vert_\Aa^2 := \int_D |f(z)|^2 (1-|z|^2)^\alpha \,dz <+\infty.
\]
Analogously, the Bergman space of the upper half place $\Aa(\bC^+)$ is defined as the set of analytic functions in $\bC^+$ such that
\[
\|f\|_{\Aa(\bC^+)}^2 = \int_{\bC^+} |f(z)|^2 s^{\alpha} \, d\mu^+(z),
\]
where $d \mu^+$ stands for the normalized area measure on $\bC^+.$ These two spaces defined above do not only share similarities in their definition, but indeed it can be shown that they are \emph{isomorphic:} if one defines
\[
T_{\alpha}f(w) = \frac{2^{\alpha/2}}{(1-w)^{\alpha+2}} f \left(\frac{w+1}{i(w-1)} \right),
\]
then $T_{\alpha}$ maps $\Aa(\bC^+)$ to $\Aa(D)$ as a \emph{unitary isomorphism.} For this reason, dealing with one space or the other is equivalent, an important fact in the proof of the main theorem below.

For the reason above, let us focus on the case of $D$, and thus we abbreviate $\Aa(D) = \Aa$ from now on. The weighted $L^2$ norm defining this space is induced by the scalar product
\[
\langle f,g\rangle_\alpha := \int_D f(z)\overline{g(z)} (1-|z|^2)^\alpha\, dz.
\]
Here and throughout, $dz$ denotes the bidimensional Lebesgue measure on $D$.

An orthonormal basis of $\Aa$ is given by the normalized monomials
$ z^n/\sqrt{c_n}$ ($n=0,1,2,\ldots$), where
\[
c_n = \int_D |z|^{2n}(1-|z|^2)^\alpha \,dz=
2\pi \int_0^1 r^{2n+1}(1-r^2)^\alpha\,dr=
\frac{\Gamma(\alpha+1)\Gamma(n+1)}{\Gamma(2+\alpha+n)}\pi.
\]
Notice that
\[
\frac 1 {c_n}=\frac {(\alpha+1)(\alpha+2)\cdots (\alpha+n+1)}{\pi n!}
=\frac{\alpha+1}\pi \binom {-\alpha-2}{n}(-1)^n ,
\]
so that from the binomial series we obtain
\begin{equation}
\label{seriescn} \sum_{n=0}^\infty \frac {x^n}{c_n}=\frac{\alpha+1}\pi (1-x)^{-2-\alpha},\quad x\in D.
\end{equation}
Given $w\in D$, the reproducing kernel relative to $w$, i.e. the (unique) function
$K_w\in\Aa$ such that
\begin{equation}
\label{repker}
f(w)=\langle f,K_w\rangle_\alpha\quad\forall f\in\Aa,
\end{equation}
is given by
\[
K_w(z):=\frac {1+\alpha}\pi (1-\overline{w}z)^{-\alpha-2}=
\sum_{n=0}^\infty \frac{\overline{w}^n z^n}{c_n},\quad z\in D
\]
(the second equality follows from \eqref{seriescn}; note that $K_w\in\Aa$, since
the sequence $\overline{w}^n/\sqrt{c_n}$ of its coefficients w.r.to the monomial
basis belongs to $\ell^2$). To see that \eqref{repker} holds, it suffices to check it when
$f(z)=z^k$ for some $k\geq 0$, but this is immediate from the series representation
of $K_w$, i.e.
\[
\langle z^k,K_w\rangle_\alpha
=\sum_{n=0}^\infty w^n \langle z^k,z^n/c_n\rangle_\alpha=w^k=f(w).
\]
Concerning the norm of $K_w$, we have readily from the reproducing property the following identity concerning their norms:
\[
\Vert K_w\Vert_\Aa^2=\langle K_w,K_w\rangle_\alpha= K_w(w)=\frac{1+\alpha}\pi (1-|w|^2)^{-2-\alpha}.
\]
We refer the reader to \cite{Seip} and the references therein for further meaningful properties in the context of Bergman spaces. 

\subsection{The Bergman transform} Now, we shall connect the first two subsections above by relating the wavelet transform to Bergman spaces, through the so-called \emph{Bergman transform.} For more detailed information, see, for instance \cite{Abreu} or \cite[Section~4]{AbreuDoerfler}. 

Indeed, fix $\alpha > -1.$ Recall that the function $\psi_{\alpha} \in H^2(\bC^+)$ satisfies
\[
\widehat{\psi_{\alpha}} = \frac{1}{c_{\alpha}} 1_{[0,+\infty)} t^{\alpha} e^{-t},
\]
where $c_{\alpha} > 0$ is chosen so that $\| \widehat{\psi_{\alpha}} \|_{L^2(\bR^+,t^{-1})}^2 =1.$ The \emph{Bergman transform of order $\alpha$} is then given by
\[
B_{\alpha}f(z) = \frac{1}{s^{\frac{\alpha}{2} +1}} W_{\overline{\psi_{\frac{\alpha+1}{2}}}} f(-x,s) = c_{\alpha} \int_0^{+\infty} t^{\frac{\alpha+1}{2}} \widehat{f}(t) e^{i z t} \, dx.
\]
From this definition, it is immediate that $B_{\alpha}$ defines an analytic function whenever $f \in H^2(\bC^+).$ Moreover, it follows directly from the properties of the wavelet transform above that $B_{\alpha}$ is a unitary map between $H^2(\bC^+)$ and $\Aa(\bC^+).$

Finally, note that the Bergman transform $B_{\alpha}$ is actually an \emph{isomorphism} between $H^2(\bC^+)$ and $\Aa(\bC^+).$

Indeed, let  $l_n^{\alpha}(x) = 1_{(0,+\infty)}(x) e^{-x/2} x^{\alpha/2} L_n^{\alpha}(x),$ where $\{L_n^{\alpha}\}_{n \ge 0}$ is the sequence of generalized Laguerre polynomials of order $\alpha.$ It can be shown that the function $\psi_n^{\alpha}$ so that
\begin{equation}\label{eq:eigenfunctions}
\widehat{\psi_n^{\alpha}}(t) = b_{n,\alpha} l_n^{\alpha}(2t),
\end{equation}
with $b_{n,\alpha}$ chosen for which $ \|\widehat{\psi_n^{\alpha}}\|_{L^2(\bR^+,t^{-1})}^2=1,$ satisfies
\begin{equation}\label{eq:eigenfunctions-disc}
T_{\alpha} (B_{\alpha}\psi_n^{\alpha}) (w) = e_n^{\alpha}(w).
\end{equation}
Here, $e_n^{\alpha}(w) = d_{n,\alpha} w^n,$ where $d_{n,\alpha}$ is so that $\|e_n^{\alpha}\|_{\Aa} = 1.$ Thus, $T_{\alpha} \circ B_{\alpha}$ is an isomorphism between $H^2(\bC^+)$ and $\Aa(D),$ and the claim follows.
\section{The main inequality}

\subsection{Reduction to an optimisation problem on Bergman spaces} By the definition of the Bergman transform above, we see that
\[
\int_{\Delta} |W_{\overline{\psi_{\beta}}} f(x,s)|^2 \, \frac{ dx \, ds}{s^2} = \int_{\tilde{\Delta}} |B_{\alpha}f(z)|^{2} s^{\alpha} \, dx \, ds,
\]
where $\tilde{\Delta} =\{ z = x + is\colon -x+is \in \Delta\}$ and $\alpha = 2\beta - 1.$ On the other hand, we may further apply the map $T_{\alpha}$ above to $B_{\alpha}f;$ this implies that
\[
\int_{\tilde{\Delta}} |B_{\alpha}f(z)|^{2} s^{\alpha} \, dx \, ds = \int_{\Omega} |T_{\alpha}(B_{\alpha}f)(w)|^2 (1-|w|^2)^{\alpha} \, dw,
\]
where $\Omega$ is the image of $\tilde{\Delta}$ under the map $z \mapsto \frac{z-i}{z+i}$ on the upper half plane $\bC^+.$ Notice that, from this relationship, we have
\begin{align*}
  & \int_{\Omega} (1-|w|^2)^{-2} \, dw  = \int_D 1_{\Delta}\left( \frac{w+1}{i(w-1)} \right) (1-|w|^2)^{-2} \, dw \cr
 & = \frac{1}{4} \int_{\Delta} \frac{ dx \, ds}{s^2} = \frac{\nu(\Delta)}{4}. \cr
\end{align*}

This leads us naturally to consider, on the disc $D$,  the Radon measure
\[
\mu(\Omega):=\int_\Omega (1-|z|^2)^{-2}dz,\quad\Omega\subseteq D,
\]
which is, by the computation above, the area measure in the usual Poincar\'e
model of the hyperbolic space (up to a multiplicative factor 4).  Thus, studying the supremum of $C_{\Delta}^{\beta}$ over $\Delta$ for which $\nu(\Delta) = s$ is equivalent to maximising
\begin{equation}\label{eq:optimal-bergman-object}
R(f,\Omega)=
\frac{\int_\Omega |f(z)|^2 (1-|z|^2)^\alpha \,dz}{\Vert f\Vert_\Aa^2}
\end{equation}
over all $f \in \Aa$ and $\Omega \subset D$ with $\mu(\Omega) = s/4.$

With these reductions, we are now ready to state a more precise version of Theorem \ref{thm:main}.

\begin{theorem}\label{thm:main-bergman} Let $\alpha>-1,$ and $s>0$ be fixed. Among all functions $f\in \Aa$ and among
all measurable sets $\Omega\subset D$ such that $\mu(\Omega)=s$, the quotient $R(f,\Omega)$ as defined in \eqref{eq:optimal-bergman-object} satisfies the inequality
\begin{equation}\label{eq:upper-bound-quotient}
	R(f,\Omega) \le R(1,D_s),
\end{equation}
where $D_s$ is a disc centered at the origin with $\mu(D_s) = s.$ Moreover, there is equality in \eqref{eq:upper-bound-quotient} if and only if $f$ is a multiple of some
reproducing kernel $K_w$ and $\Omega$ is a ball centered at
$w$, such that $\mu(\Omega)=s$.
\end{theorem}

Note that, in the Poincar\'e disc model in two dimensions, balls in the pseudohyperbolic metric coincide with Euclidean balls, but the Euclidean and hyperbolic centers differ in general, as well as the respective radii. 

\begin{proof}[Proof of Theorem \ref{thm:main-bergman}] Let us begin by computing $R(f,\Omega)$ when $f=1$ and $\Omega=B_r(0)$ for some $r<1$.
\[
R(1,B_r)=\frac {\int_0^r \rho (1-\rho^2)^\alpha\,d\rho}
{\int_0^1 \rho (1-\rho^2)^\alpha\,d\rho}
= \frac {(1-\rho^2)^{1+\alpha}\vert_0^r}
{(1-\rho^2)^{1+\alpha}\vert_0^1}
=1-(1-r^2)^{1+\alpha}.
\]
Since $\mu(B_r)$ is given by
\begin{align*}
\int_{B_r} (1-|z|^2)^{-2}\,dz
 & =2\pi \int_0^r \rho (1-\rho^2)^{-2}\,d\rho \cr
=\pi(1-r^2)^{-1}|_0^r & =\pi\left(\frac{1}{1-r^2}-1\right), \cr
\end{align*}
we have
\[
\mu(B_r)=s \iff \frac 1{1-r^2}=1+\frac s\pi,
\]
so that $\mu(B_r)=s$ implies $R(1,B_r)=1-(1+s/\pi)^{-1-\alpha}.$ The function
\[
\theta(s):=1-(1+s/\pi)^{-1-\alpha},\quad s\geq 0
\]
will be our comparison function, and we will prove that
\[
R(f,\Omega)\leq \theta(s)
\]
for every $f$ and every $\Omega\subset D$ such that $\mu(\Omega)=s$.

Consider any $f\in\Aa$ such that $\Vert f\Vert_\Aa=1$, let
\[
u(z):= |f(z)|^2 (1-|z|^2)^{\alpha+2},
\]
and observe that
\begin{equation}
\label{eq10}
R(f,\Omega)=\int_\Omega u(z)\,d\mu
\leq I(s):=\int_{\{u>u^*(s)\}} u(z) \,d\mu,\quad
s=\mu(\Omega),
\end{equation}
where $u^*(s)$ is the unique value of $t>0$ such that
\[
\mu(\{u>t\})=s.
\]
That is, $u^*(s)$ is the inverse function of the distribution function of $u$,
relative to the measure $\mu$.

Observe that $u(z)$ can be extended to a continuous function on $\overline D$,
by letting $u\equiv 0$ on $\partial D.$

Indeed, consider any $z_0\in D$ such that,
say, $|z_0|>1/2$, and let $r=(1-|z_0|)/2$. Then, on the ball $B_r(z_0)$, for some
universal constant $C>1$ we have
\[
C^{-1} (1-|z|^2) \leq r \leq C(1-|z|^2)\quad\forall z\in B_r(z_0),
\]
so that
\begin{align*}
	\omega(z_0):=\int_{B_r(z_0)} |f(z)|^2 (1-|z|^2)^\alpha \,dz
	\geq C_1 r^{\alpha+2}\frac 1 {\pi r^2}
	\int_{B_r(z_0)} |f(z)|^2 \,dz\\
	\geq C_1 r^{\alpha+2} |f(z_0)|^2
	\geq C_2 (1-|z_0|^2)^{\alpha+2} |f(z_0)|^2=
	C_2 u(z_0).
\end{align*}
Here, we used that fact that $|f(z)|^2$ is subharmonic, which follows from analyticity. Since $|f(z)|^2 (1-|z|^2)^\alpha\in L^1(D)$,
$\omega(z_0)\to 0$ as $|z_0|\to 1$, so that
\[
\lim_{|z_0|\to 1} u(z_0)=0.
\]
As a consequence, we obtain that the superlevel sets $\{u > t\}$ are \emph{strictly} contained in $D$. Moreover, the function $u$ so defined is a \emph{real analytic function}. Thus (see \cite{KrantzParks}) all level sets of $u$ have zero measure, and as all superlevel sets do not touch the boundary, the hyperbolic length of all level sets is zero; that is,
\[
L(\{u=t\}) := \int_{\{u = t\}} (1-|z|^2)^{-1} \, d\mathcal{H}^1 =0, \, \forall \, t > 0.
\]
Here and throughout the proof, we use the notation $\mathcal{H}^k$ to denote the $k-$dimensional Hausdorff measure.  

It also follow from real analyticity that the set of critical points of $u$ also has hyperbolic length zero:
\[
L(\{|\nabla u| = 0\}) = 0.
\]
Finally, we note that a suitable adaptation of the proof of Lemma 3.2 in \cite{NicolaTilli} yields the following result. As the proofs are almost identical, we omit them, and refer the interested reader to the original paper.

\begin{lemma}\label{thm:lemma-derivatives} The function $\varrho(t) := \mu(\{ u > t\})$ is absolutely continuous on $(0,\max u],$ and
	\[
	-\varrho'(t) = \int_{\{u = t\}} |\nabla u|^{-1} (1-|z|^2)^{-2} \, d \mathcal{H}^1.
	\]
	In particular, the function $u^*$ is, as the inverse of $\varrho,$ locally absolutely continuous on $[0,+\infty),$ with
	\[
	-(u^*)'(s) = \left( \int_{\{u=u^*(s)\}} |\nabla u|^{-1} (1-|z|^2)^{-2} \, d \mathcal{H}^1 \right)^{-1}.
	\]

\end{lemma}

Let us then denote the boundary of the superlevel set where $u > u^*(s)$ as
\[
A_s=\partial\{u>u^*(s)\}.
\]
We have then, by Lemma \ref{thm:lemma-derivatives},
\[
I'(s)=u^*(s),\quad
I''(s)=-\left(\int_{A_s} |\nabla u|^{-1}(1-|z|^2)^{-2}\,d{\mathcal H}^1\right)^{-1}.
\]
Since the Cauchy-Schwarz inequality implies
\[
\left(\int_{A_s} |\nabla u|^{-1}(1-|z|^2)^{-2}\,d{\mathcal H}^1\right)
\left(\int_{A_s} |\nabla u|\,d{\mathcal H}^1\right)
\geq
\left(\int_{A_s} (1-|z|^2)^{-1}\,d{\mathcal H}^1\right)^2,
\]
letting
\[
L(A_s):= \int_{A_s} (1-|z|^2)^{-1}\,d{\mathcal H}^1
\]
denote the length of $A_s$ in the hyperbolic metric, we obtain the lower bound
\begin{equation}\label{eq:lower-bound-second-derivative}
I''(s)\geq - \left(\int_{A_s} |\nabla u|\,d{\mathcal H}^1\right)
L(A_s)^{-2}.
\end{equation}
In order to compute the first term in the product on the right-hand side of \eqref{eq:lower-bound-second-derivative}, we first note that
\[
\Delta \log u(z) =\Delta \log (1-|z|^2)^{2 + \alpha}=-4(\alpha+2)(1-|z|^2)^{-2},
\]
which then implies that, letting $w(z)=\log u(z)$,
\begin{align*}
\frac {-1} {u^*(s)} \int_{A_s} |\nabla u|\,d{\mathcal H}^1
 & =
\int_{A_s} \nabla w\cdot\nu \,d{\mathcal H}^1
=
\int_{u>u^*(s)} \Delta w\,dz \cr
=-4(\alpha+2)\int_{u>u^*(s)}
(1-|z|^2)^{-2}
\,dz
& =-4(\alpha+2) \mu(\{u>u^*(s)\})=
-4(\alpha+2)s.\cr
\end{align*}
Therefore,
\begin{equation}\label{eq:lower-bound-second-almost}
I''(s)\geq -4(\alpha+2)s u^*(s)L(A_s)^{-2}=
-4(\alpha+2)s I'(s)L(A_s)^{-2}.
\end{equation}
On the other hand, the isoperimetric inequality on the Poincaré disc - see, for instance, \cite{Izmestiev, Osserman, Schmidt} - implies
\[
L(A_s)^2 \geq 4\pi s + 4 s^2,
\]
so that, pluggin into \eqref{eq:lower-bound-second-almost}, we obtain
\begin{equation}\label{eq:final-lower-bound-second}
I''(s)\geq -4 (\alpha+2)s I'(s)(4\pi s+4 s^2)^{-1}
=-(\alpha+2)I'(s)(\pi+s)^{-1}.
\end{equation}
Getting back to the function $\theta(s)$, we have
\[
\theta'(s)=\frac{1+\alpha}\pi(1+s/\pi)^{-2-\alpha},\quad
\theta''(s)=-(2+\alpha)\theta'(s)(1+s/\pi)^{-1}/\pi.
\]
Since
\[
I(0)=\theta(0)=0\quad\text{and}\quad \lim_{s\to+\infty} I(s)=\lim_{s\to+\infty}\theta(s)=1,
\]
we may obtain, by a maximum principle kind of argument,
\begin{equation}\label{eq:inequality-sizes}
I(s)\leq\theta(s)\quad\forall s>0.
\end{equation}
Indeed, consider $G(s) := I(s) - \theta(s).$ We claim first that $G'(0) \le 0.$ To that extent, notice that
\[
\Vert u\Vert_{L^\infty(D)} = u^*(0)=I'(0) \text{ and }\theta'(0)=\frac{1+\alpha}\pi.
\]
On the other hand, we have, by the properties of the reproducing kernels,
\begin{align}\label{eq:sup-bound}
	u(w)=|f(w)|^2 (1-|w|^2)^{\alpha+2}& =|\langle f,K_w\rangle_\alpha|^2(1-|w|^2)^{\alpha+2}\cr
	\leq \Vert f\Vert_\Aa^2 \Vert K_w\Vert_\Aa^2&  (1-|w|^2)^{\alpha+2}=\frac{1+\alpha}\pi,
\end{align}
and thus $I'(0) - \theta'(0) \le 0,$ as claimed. Consider then
\[
m := \sup\{r >0 \colon G \le 0 \text{ over } [0,r]\}.
\]
Suppose $m < +\infty.$ Then, by compactness, there is a point $c \in [0,m]$ so that $G'(c) = 0,$ as $G(0) = G(m) = 0.$ Let us first show that $G(c)<0$ if $G \not\equiv 0.$

In fact, we first define the auxiliary function $h(s) = (\pi + s)^{\alpha + 2}.$ The differential inequalities that $I, \, \theta$ satisfy may be combined, in order to write
\begin{equation}\label{eq:functional-inequality}
(h \cdot G')' \ge 0.
\end{equation}
Thus, $h\cdot G'$ is increasing on the whole real line. As $h$ is increasing on $\bR,$ we have two options:
\begin{enumerate}
	\item either $G'(0) = 0,$ which implies, from \eqref{eq:sup-bound}, that $f$ is a multiple of the reproducing kernel $K_w.$ In this case, It can be shown that $G \equiv 0,$ which contradicts our assumption;
	\item or $G'(0)<0,$ in which case the remarks made above about $h$ and $G$ imply that $G'$ is \emph{increasing} on the interval $[0,c].$ In particular, as $G'(c) =0,$ the function $G$ is \emph{decreasing} on $[0,c],$ and the claim follows.
\end{enumerate}

Thus, $c \in (0,m).$ As $G(m) = \lim_{s \to \infty} G(s) = 0,$ there is a point $c' \in [m,+\infty)$ so that $G'(c') = 0.$ But this is a contradiction to \eqref{eq:functional-inequality}: notice that $0 = G(m) > G(c)$ implies the existence of a point $d \in (c,m]$ with $G'(d) > 0.$ As $h \cdot G'$ is increasing over $\bR,$ and $(h \cdot G')(c) = 0, \, (h \cdot G')(d) > 0,$ we cannot have $(h \cdot G') (c') = 0.$ The contradiction stems from supposing that $m < +\infty,$ and \eqref{eq:inequality-sizes} follows.

With \eqref{eq:upper-bound-quotient} proved, we now turn our attention to analysing the equality case in Theorem \ref{thm:main-bergman}. To that extent, notice that, as a by-product of the analysis above, the inequality \eqref{eq:inequality-sizes} is \emph{strict} for every $s>0,$ unless $I\equiv\theta$. Now assume that
$I(s_0)=\theta(s_0)$ for some $s_0>0$, then $\Omega$ must coincide (up to a negligible set)
with $\{u>u^*(s_0)\}$ (otherwise we would have strict inequality in
\eqref{eq10}), and moreover $I\equiv \theta$, so that
\[
\Vert u\Vert_{L^\infty(D)} = u^*(0)=I'(0)=\theta'(0)=\frac{1+\alpha}\pi.
\]
By the argument above in \eqref{eq:sup-bound}, this implies that the $L^\infty$ norm of $u$ on $D$, which is equal to $(1+\alpha)/\pi$,
 is attained at some $w\in D$, and since equality is achieved, we obtain that $f$ must be a multiple of the reproducing kernel $K_w$, as desired. This concludes the proof of Theorem \ref{thm:main-bergman}.
\end{proof}

\noindent\textbf{Remark 1.} The uniqueness part of Theorem \ref{thm:main-bergman} may also be analysed through the lenses of an overdetermined problem. In fact, we have equality in that result if and only if we have equality in \eqref{eq:final-lower-bound-second}, for almost every $s > 0.$ If we let $w = \log u$, then a quick inspection of the proof above shows that \begin{align}\label{eq:serrin-disc}
	\begin{cases}
	\Delta w  = \frac{-4(\alpha+2)}{(1-|z|^2)^2} &  \text { in }  \{u > u^*(s)\}, \cr
	 w  = \log u^*(s), & \text{ on } A_s, \cr
	 |\nabla w| = \frac{c}{1-|z|^2}, & \text{ on } A_s. \cr
	\end{cases}
\end{align}
By mapping the upper half plane $\mathbb{H}^2$ to the Poincar\'e disc by $z \mapsto \frac{z-i}{z+i},$ one sees at once that a solution to \eqref{eq:serrin-disc} translates into a solution of the Serrin overdetermined problem
\begin{align}\label{eq:serrin-upper-half}
	\begin{cases}
		\Delta_{\mathbb{H}^2} v  = c_1 &  \text { in }  \Omega, \cr
		v = c_2 & \text{ on } \partial\Omega, \cr
		|\nabla_{\mathbb{H}^2} v| = c_3 & \text{ on } \partial\Omega, \cr
	\end{cases}
\end{align}
where $\Delta_{\mathbb{H}^2}$ and $\nabla_{\mathbb{H}^2}$ denote, respectively, the Laplacian and gradient in the upper half space model of the two-dimensional hyperbolic plane. By the main result in \cite{KumaresanPrajapat}, the only domain $\Omega$ which solves \eqref{eq:serrin-upper-half} is a geodesic disc in the upper half space, with the hyperbolic metric. Translating back, this implies that $\{u>u^*(s)\}$ are (hyperbolic) balls for almost all $s > 0.$ A direct computation then shows that $w = \log u,$ with $u(z) = |K_w(z)|^2(1-|z|^2)^{\alpha+2},$ is the unique solution to \eqref{eq:serrin-disc} in those cases.   \\

\noindent\textbf{Remark 2.} Theorem \ref{thm:main-bergman} directly implies, by the reductions above, Theorem \ref{thm:main}. In addition to that, we may use the former to characterise the extremals to the inequality \eqref{eq:first-theorem}.

Indeed, it can be shown that the reproducing kernels $K_w$ for $\Aa(D)$ are the image under $T_{\alpha}$ of the reproducing kernels for $\Aa(\bC^+),$ given by
\[
\mathcal{K}_{w}^{\alpha}(z) = \kappa_{\alpha} \left( \frac{1}{z-\overline{w}} \right)^{\alpha+2},
\]
where $\kappa_{\alpha}$ accounts for the normalisation we used before. Thus, equality holds in \eqref{eq:first-theorem} if and only if $\Delta$ is a pseudohyperbolic disc, and moreover, the function $f \in H^2(\bC^+)$ is such that
\begin{equation}\label{eq:equality-Bergman-kernel}
B_{2\beta-1}f(z) = \lambda_{\beta} \mathcal{K}^{2\beta - 1}_w(z),
\end{equation}
for some $w \in \bC^+.$ On the other hand, it also holds that the functions $\{\psi^{\alpha}_n\}_{n \in \bN}$ defined in \eqref{eq:eigenfunctions} are so that $B_{\alpha}(\psi_0^{\alpha}) =: \Psi_0^{\alpha}$ is a \emph{multiple} of $\left(\frac{1}{z+i}\right)^{\alpha+2}.$ This can be seen by the fact that $T_{\alpha}(\Psi_0^{\alpha})$ is the constant function.

 From these considerations, we obtain that $f$ is a multiple of $\pi_{w} \psi_0^{2\beta-1},$ where $\pi_w$ is as in \eqref{eq:wavelet-def}. In summary, we obtain the following:

 \begin{corollary} Equality holds in Theorem \ref{thm:main} if an only if $\Delta$ is a pseudohyperbolic disc with hyperbolic center $w = x + i y,$ and $$f(t) = c \cdot \frac{1}{y^{1/2}}\psi_0^{2\beta-1} \left( \frac{t-x}{y}\right),$$
 for some $c \in \mathbb{C} \setminus \{0\}.$
 \end{corollary}

\section{Other measure contraints and related problems} As discussed in the introduction, the constraint on the \emph{hyperbolic} measure of the set $\Delta$ can be seen as the one which makes the most sense in the framework of the Wavelet transform.

In fact, another way to see this is as follows. Fix $w = x_1 + i s_1,$ and let $z = x + is, \,\, w,z \in \bC^+.$ Then
\[
\langle \pi_{w} f, \pi_z g \rangle_{H^2(\bC^+)} = \langle f, \pi_{\tau_{w}(z)} g \rangle_{H^2(\bC^+)},
\]
where we define $\tau_{w}(z) = \left( \frac{x-x_1}{s_1}, \frac{s}{s_1}  \right).$ By \eqref{eq:wavelet-def}, we get
\begin{align}\label{eq:change-of-variables}
\int_{\Delta} |W_{\overline{\psi_{\beta}}}(\pi_w f)(x,s)|^2 \, \frac{ dx \, ds}{s^2} & = \int_{\Delta} |W_{\overline{\psi_{\beta}}}f(\tau_w(z))|^2 \, \frac{dx \, ds}{s^2} \cr
& = \int_{(\tau_w)^{-1}(\Delta)} |W_{\overline{\psi_{\beta}}}f(x,s)|^2 \, \frac{dx \, ds}{s^2}. \cr
\end{align}
Thus, suppose one wants to impose a measure constraint like $\tilde{\nu}(\Delta) = s,$ where $\tilde{\nu}$ is a measure on the upper half plane. The computations in \eqref{eq:change-of-variables} tell us that $C_{\Delta}^{\beta} = C_{\tau_w(\Delta)}^{\beta}, \, \forall \, w \in \bC^+.$ Thus, one is naturally led to suppose that the class of domains $\{ \tilde{\Delta} \subset \bC^+ \colon \tilde{\nu}(\tilde{\Delta}) = \tilde{\nu}(\Delta) \}$ includes $\{ \tau_w(\Delta), \, w \in \bC^+.\}.$

Therefore, $\tilde{\nu}(\Delta) = \tilde{\nu}(\tau_w(\Delta)).$ Taking first $w = x_1 + i,$ one obtains that $\tilde{\nu}$ is invariant under horizontal translations. By taking $w = is_1,$ one then obtains that $\tilde{\nu}$ is invariant with respect to (positive) dilations. It is easy to see that any measure with these properties has to be a multiple of the measure $\nu$ defined above.

On the other hand, if one is willing to forego the original problem and focus on the quotient \eqref{eq:optimal-bergman-object}, one may wonder what happens when, instead of the hyperbolic measure on the (Poincar\'e) disc, one considers the supremum of $R(f,\Omega)$ over $f \in \Aa(D)$, and now look at $|\Omega| =s,$ where $| \cdot |$ denotes \emph{Lebesgue} measure.

In that case, the problem of determining
\[
\mathcal{C}_{\alpha} := \sup_{|\Omega| = s} \sup_{f \in \Aa(D)} R(f,\Omega)
\]
is much simpler. Indeed, take $\Omega = D \setminus D(0,r_s),$ with $r_s > 0$ chosen so that the Lebesgue measure constraint on $\Omega$ is satisfied. For such a domain, consider $f_n(z) = d_{n,\alpha} \cdot z^n,$ as in  \eqref{eq:eigenfunctions-disc}. One may compute these constants explicitly as:
\[
d_{n,\alpha} = \left( \frac{\Gamma(n+2+\alpha)}{n! \cdot \Gamma(2+\alpha)} \right)^{1/2}.
\]
For these functions, one has $\|f_n\|_{\Aa} = 1.$ We now claim that
\begin{equation}\label{eq:convergence-example}
\int_{D(0,r_s)} |f_n(z)|^2(1-|z|^2)^{\alpha} \, dz \to 0 \text{ as } n \to \infty.
\end{equation}
Indeed, the left-hand side of \eqref{eq:convergence-example} equals, after polar coordinates,
\begin{equation}\label{eq:upper-bound}
2 \pi d_{n,\alpha}^2 \int_0^{r_s} t^{2n} (1-t^2)^{\alpha} \, dt \le 2 \pi d_{n,\alpha}^2 (1-r_s^2)^{-1} r_s^{2n},
\end{equation}
whenever $\alpha > -1.$ On the other hand, the explicit formula for $d_{n,\alpha}$ implies this constant grows at most like a (fixed) power of $n.$ As the right-hand side of \eqref{eq:upper-bound} contains a $r_s^{2n}$ factor, and $r_s < 1,$ this proves \eqref{eq:convergence-example}. Therefore,
\[
R(f_n,\Omega) \to 1 \text{ as } n \to \infty.
\]
So far, we have been interested in analysing the supremum of $\sup_{f \in \Aa} R(f,\Omega)$ over different classes of domains, but another natural question concerns a \emph{reversed} Faber-Krahn inequality: if one is instead interested in determining the \emph{minimum} of$\sup_{f \in \Aa} R(f,\Omega)$ over certain classes of domains, what can be said in both Euclidean and hyperbolic cases?

 In that regard, we first note the following: the problem of determining the \emph{minimum} of $\sup_{f \in \Aa} R(f,\Omega)$ over $\Omega \subset D, \, \mu(\Omega) = s$ is much easier than the analysis in the proof of Theorem \ref{thm:main-bergman} above.  Indeed, by letting $\Omega_n$ be a sequence of annuli of hyperbolic measure $s,$ one sees that $\sup_{f \in \Aa} R(f,\Omega_n) = R(1,\Omega_n), \, \forall n \in \bN,$ by the results in \cite{DaubechiesPaul}. Moreover, if $\mu(\Omega_n) = s,$ one sees that we may take $\Omega_n \subset D \setminus D\left(0,1-\frac{1}{n}\right), \, \forall n \ge 1,$ and thus $|\Omega_n| \to 0 \, \text{ as } n \to \infty.$ This shows that
 \[
 \inf_{\Omega \colon \mu(\Omega) = s}  \sup_{f \in \Aa(D)} R(f,\Omega) = 0, \, \forall \, \alpha > -1.
 \]
On the other hand, the situation is starkly different when one considers the Lebesgue measure in place of the hyperbolic one. Indeed, we shall show below that we may also explicitly solve the problem of determining the \emph{minimum} of $\sup_{f \in \Aa} R(f,\Omega)$ over all $\Omega, \, |\Omega| = s.$ For that purpose, we define
\[
\mathcal{D}_{\alpha} = \inf_{\Omega\colon |\Omega| = s} \sup_{f \in \Aa} R(f,\Omega).
\]
Then we have
\begin{equation}\label{eq:lower-bound}
\mathcal{D}_{\alpha} \ge \inf_{|\Omega| = s} \frac{1}{\pi} \int_{\Omega} (1-|z|^2)^{\alpha} \, dz.
\end{equation}
Now, we have some possibilities:
\begin{enumerate}
	\item If $\alpha \in (-1,0),$ then the function $z \mapsto (1-|z|^2)^{\alpha}$ is strictly \emph{increasing} on $|z|,$ and thus the left-hand side of \eqref{eq:lower-bound} is at least
	\[
	2  \int_0^{(s/\pi)^{1/2}} t(1-t^2)^{\alpha} \, dt = \theta^1_{\alpha}(s).
	\]
	\item If $\alpha > 0,$ then the function $z \mapsto (1-|z|^2)^{\alpha}$ is strictly \emph{decreasing} on $|z|,$ and thus the left-hand side of \eqref{eq:lower-bound} is at least
	\[
	2 \int_{(1-s/\pi)^{1/2}}^1 t(1-t^2)^{\alpha} \, dt = \theta^2_{\alpha}(s).
	\]
	\item Finally, for $\alpha = 0,$ $\mathcal{D}_0 \ge s.$
\end{enumerate}

In particular, we can also characterise \emph{exactly} when equality occurs in the first two cases above: for the first case, we must have $\Omega = D(0,(s/\pi)^{1/2});$ for the second case, we must have $\Omega = D \setminus D(0,(1-s/\pi)^{1/2});$ notice that, in both those cases, equality is indeed attained, as constant functions do indeed attain  $\sup_{f \in \Aa} R(f,\Omega).$

Finally, in the third case, if one restricts to \emph{simply connected sets} $\Omega \subset D,$ we may to resort to \cite[Theorem~2]{AbreuDoerfler}.

Indeed, in order for the equality $\sup_{f \in \mathcal{A}_0} R(f,\Omega) = \frac{|\Omega|}{\pi},$ to hold, one necessarily has
\[
\mathcal{P}(1_{\Omega}) = \lambda,
\]
where $\mathcal{P}: L^2(D) \to \mathcal{A}_0(D)$ denotes the projection onto the space $\mathcal{A}_0.$ But from the proof of Theorem 2 in  \cite{AbreuDoerfler},  as $\Omega$ is simply connected, this implies that $\Omega$ has to be a disc centered at the origin. We summarise the results obtained in this section below, for the convenience of the reader.

\begin{theorem}\label{thm:sup-inf} Suppose $s = |\Omega|$ is fixed, and consider $\mathcal{C}_{\alpha}$ defined above. Then $C_{\alpha} =1, \forall \alpha > -1,$ and no domain $\Omega$ attains this supremum.
	
Moreover, if one considers $ \mathcal{D}_{\alpha},$ one has the following assertions:
\begin{enumerate}
	\item If $\alpha \in (-1,0),$ then $\sup_{f \in \Aa} R(f,\Omega) \ge \theta_{\alpha}^1(s),$ with equality if and only if $\Omega = D(0,(s/\pi)^{1/2}).$
	\item If $\alpha > 0,$ then $\sup_{f \in \Aa} R(f,\Omega) \ge \theta_{\alpha}^2(s),$ with equality if and only if $\Omega = D \setminus D(0,(1-s/\pi)^{1/2}).$
	\item If $\alpha = 0,$ $\sup_{f \in \Aa} R(f,\Omega) \ge s.$ Furthermore, if $\Omega$ is simply connected, then $\Omega = D(0,(s/\pi)^{1/2}).$
\end{enumerate}
\end{theorem}

The assuption that $\Omega$ is simply connected in the third assertion in Theorem \ref{thm:sup-inf} cannot be dropped in general, as any radially symmetric domain $\Omega$ with Lebesgue measure $s$ satisfies the same property. We conjecture, however, that these are the \emph{only} domains with such a property: that is, if $\Omega$ is such that $\sup_{f \in \mathcal{A}_0} R(f,\Omega) = |\Omega|,$ then $\Omega$ must have radial symmetry.

\end{document}